\documentclass[11pt]{article} 
\usepackage{a4,fullpage,amssymb,epsf,psfrag,times}

\makeatletter\@addtoreset{equation}{section}\makeatother
\makeatletter\@addtoreset{figure}{section}\makeatother
\makeatletter\@addtoreset{table}{section}\makeatother

\newtheorem{theorem}{Theorem}[section]
\newtheorem{proposition}[theorem]{Proposition}
\newtheorem{lemma}[theorem]{Lemma}

\newtheorem{corollary}[theorem]{Corollary}

\newcommand{\R}{{\mathbb R}}
\newcommand{\C}{{\mathbb C}}
\newcommand{\Z}{{\mathbb Z}}

\newcommand{\T}{{\mathbb T}}

\newcommand{\B}{{\mathbb B}}

\newcommand{\Proj}{{\mathbb P}}

\newcommand{\op}[1]{\!\!\mathop{\rm ~#1}\nolimits}
\newcommand{\fop}[1]{\!\!\mathop{\mbox{\rm \footnotesize ~#1}}\nolimits}
\newcommand{\scriptop}[1]{\!\!\mathop{\mbox{\rm \scriptsize ~#1}}\nolimits}

\newenvironment{proof}{\par\medskip\noindent{\bf Proof}~~}{
\unskip\nobreak\hfill\hbox{$\Box$}\par \bigskip}        

\newenvironment{remark}{\refstepcounter{theorem}\par\medskip\noindent{\bf Remark~\thetheorem~~}}{\unskip\nobreak\hfill\hbox{ $\oslash$}\par\bigskip}

\newcommand{\got}[1]{\mathfrak{#1}}

\title{Reduced phase space and toric variety coordinatizations
of Delzant spaces}
\author{J.J. Duistermaat\thanks{Research stimulated 
by a KNAW professorship} ~and A. Pelayo\thanks{Partly supported by
a Rackham Predoctoral fellowship}}
\begin{document}
\maketitle

\begin{abstract}
In this note we describe the natural 
coordinatizations of a Delzant space
defined as a reduced phase space 
(symplectic geometry view\--point)
and give explicit formulas for the 
coordinate transformations. 
For each fixed point of the torus action 
on the Delzant polytope, we have a 
maximal coordinatization of an open cell in 
the Delzant space which contains the fixed point.  
This cell is equal to the domain of definition of one 
of the natural coordinatizations of the Delzant space 
as a toric variety (complex algebraic geometry view\--point), 
and we give an explicit formula for the toric 
variety coordinates in terms of the reduced phase space 
coordinates. We use considerations in 
the maximal coordinate neighborhoods to give 
simple proofs of some of the basic facts 
about the Delzant space, as a reduced 
phase space, and as a toric variety.  
These can be viewed as a first application of 
the coordinatizations, and serve to make the presentation 
more self\--contained. 
\end{abstract}

\section{Introduction}
Let $(M,\,\sigma)$ be a smooth compact 
and connected symplectic manifold 
of dimension $2n$ and let $T$ be a torus which acts 
effectively on $(M,\,\sigma )$ by means of 
symplectomorphisms. If the action of $T$ on 
$(M,\,\sigma )$ is moreover Hamiltonian, then
$\op{dim}T \leq n$, and the image of the momentum 
mapping $\mu _T:M\to {\got t}^*$ is a convex polytope $\Delta$ 
in the dual space $\got{t}^*$ of $\got{t}$, where 
$\got{t}$ denotes the Lie algebra of $T$. In 
the maximal case when $\op{dim}T=n$, $(M, \, \sigma)$ 
is called a {\em Delzant space}. 

Delzant \cite[(*) on p. 323]{d} proved 
that in this case the polytope $\Delta$ is very special, a 
so\--called {\em Delzant polytope}, of which we 
recall the definition in Section \ref{delzantsec}. 
Furthermore Delzant \cite[Th. 2.1]{d} proved that 
two Delzant spaces are $T$\--equivariantly 
symplectomorphic if and only if their momentum mappings 
have the same image up to a translation by an element 
of ${\got t}^*$. Thirdly Delzant \cite[pp. 328, 329]{d} 
proved that for every Delzant polytope $\Delta$ 
there exists a  Delzant space such that $\mu _T(M)=\Delta$. 
This Delzant space is obtained as the {\em reduced phase space}  
for a linear Hamiltonian action of a torus $N$ on a symplectic vector space 
$E$, at a value $\lambda _N$ of the momentum mapping of the 
Hamiltonian $N$\--action, where $E$, $N$ and $\lambda _N$ 
are determined by the Delzant polytope. 

Finally Delzant \cite[Sec. 5]{d} observed that the Delzant polytope 
gives rise to a fan (= \'eventail in French), 
and that the Delzant space with Delzant polytope 
$\Delta$ is $T$\--equivariantly 
diffeomorphic to the {\em toric variety} $M^{\scriptop{toric}}$ 
defined by the fan. Here $M^{\scriptop{toric}}$ is 
a complex $n$\--dimensional complex analytic manifold, 
and the action of the real torus $T$ on $M^{\scriptop{toric}}$ 
has an extension to a complex analytic action on $M^{\scriptop{toric}}$ 
of the complexification $T_{\C}$ of $T$. 
In our description in Section \ref{torvarsec} of the toric variety 
$M^{\scriptop{toric}}$ we do not use fans. The information,  
for each vertex $v$ of $\Delta$, which codimension one faces 
of $\Delta$ contain $v$, already suffices to define $M^{\scriptop{toric}}$.  

In this note we show that the construction of the 
Delzant space $M$ as a reduced phase space leads, for every 
vertex $v$ of the Delzant polytope, to a natural 
coordinatization $\varphi _v$ of a $T$\--invariant open cell 
$M_v$ in $M$, where $M_v$ contains the unique fixed point 
$m_v$ in $M$ of the $T$\--action such that $\mu _T(m_v)=v$.  
We give an explicit construction of the inverse $\psi _v$ of 
$\varphi _v$, which is a maximal diffeomorphism in the sense 
of  Remark \ref{pelayorem}.  The construction of $\psi _v$ 
originated in an attempt to extend the equivariant symplectic 
ball embeddings from $(\B^{2n}_r, \, \sigma_0) \subset (\C^n, \, \sigma_0)$ 
into the Delzant space $(M, \, \sigma)$
in Pelayo \cite{pelayo} by maximal 
equivariant symplectomorphisms from open 
neighborhoods of the origin in $\C ^n$ into the 
Delzant space $(M, \, \sigma)$. If $v$ and $w$ are two 
different vertices, then the coordinate transformation 
$\varphi _w\circ{\varphi _v}^{-1}$
is given by the explicit formulas (\ref{inFv}), (\ref{notinFv}). 

Let $\Sigma$ be the set of all strata of the orbit type stratification of $M$ 
for the $T$\--action. Then the domain of definition $M_v$ of $\varphi _v$ 
is equal to the union of all $S\in\Sigma$ such that 
the fixed point $m_v$ belongs to the closure of $S$ in $M$, see 
Corollary \ref{stratacor}. The strata $S\in\Sigma$ are also the 
orbits in the toric variety $M^{\scriptop{toric}}\simeq M$ for the action 
of the complexification $T_{\C}$ of the real torus $T$, and 
the domain of definition $M_v$ of $\varphi _v$ 
is equal to the domain of definition of a 
natural complex analytic $T_{\C}$\--equivariant 
coordinatization $\Phi _v$ of a $T_{\C }$\--invariant 
open cell. The diffeomorphism $\Phi _v\circ {\varphi _v}^{-1}$, 
which sends the reduced phase space coordinates to the 
toric variety coordinates, maps $U_v:=\varphi _v(M_v)$ diffeomorphically 
onto a complex vector space, and 
is given by the explicit formulas (\ref{zzeta}). 

In the toric variety 
coordinates the complex structure is the standard 
one and the coordinate transformations are  
relatively simple Laurent monomial 
transformations, whereas the
symplectic form is generally given by quite 
complicated algebraic functions. On the other hand, in the 
reduced phase space coordinates the symplectic 
form is the standard one, but the coordinate transformations, 
and also the complex structure, have a more complicated 
appearance.

Let $F$ denote the set of all $d$ codimension one faces 
of $\Delta$ and, for every vertex $v$ of $\Delta$, let 
$F_v$ denote the set of all $f\in F$ such that $v\in f$. 
Note that $\# (F_v)=n$ for every vertex $v$ of $\Delta$. 
For any sets $A$ and $B$, let $A^B$ denote the set of all 
$A$\--valued functions on $B$. If $A$ is a field and 
the set $B$ is finite, then $A^B$ is a $\# (B)$\--dimensional 
vector space over $A$. One of the technical points in this 
paper is the efficient organization of proofs and formulas 
made possible by viewing the Delzant space as a reduction 
of the vector space $\C ^F$, and letting, for each vertex $v$, 
the coordinatizations $\varphi _v$ and $\Phi _v$ take 
their values in $\C ^{F_v}$. This leads to a natural 
projection $\rho _v:\C ^F\to\C ^{F_v}$ obtained by the restriction 
of functions on $F$ to $F_v\subset F$. For each 
vertex $v$ the complex vector space  $\C ^{F_v}$ is isomorphic 
to $\C ^n$, but the isomorphism depends on an enumeration 
of $F_v$, the introduction of which would lead to an unnecessary 
complication of the combinatorics. Similarly our torus $T$ is isomorphic 
to $\R ^n/\Z ^n$, but the isomorphism depends on the choice 
of a $\Z$\--basis of the integral lattice ${\got t}_{\Z}$ in the 
Lie algebra ${\got t}$ of $T$. 
As for each vertex $v$ a different $\Z$\--basis of ${\got t}_{\Z}$ 
appears, we also avoid such a choice, keeping $T$ in its 
abstract form. We hope and trust that this will not lead to 
confusion with our main references 
Delzant \cite{d}, Audin \cite{audin} and Guillemin \cite{g} 
about Delzant spaces, where $\C ^F$, each $\C ^{F_v}$, 
and $T$  is denoted as $\C ^d$, $\C ^n$, and 
$\R ^n/\Z ^n$, respectively. 

The organization of this manuscript is as follows. 
In Section \ref{delzantsec} we review the definition of the 
reduced phase Delzant space, 
and introduce the notations which will 
be convenient for our purposes. In Section \ref{Zsec}  we define the 
reduced phase space coordinatizations. In Section 
\ref{coordtransfsec} we give explicit formulas for 
the coordinate transformations and describe the 
reduced phase space Delzant space as obtained by gluing together 
bounded open subsets of $n$\--dimensional complex 
vector spaces with these coordinate transformations as the gluing 
maps. In Section \ref{torvarsec} we review the definition 
of the toric variety defined by the Delzant polytope, 
prove that the natural mapping from the reduced phase space 
to the toric variety is a diffeomorphism, and compare 
the coordinatizations of Section \ref{Zsec} with the 
natural coordinatizations of the toric variety. 
In Section \ref{examplesec} we present 
these computations for the two simplest classes of  examples,
the complex projective spaces and the Hirzebruch surfaces.

\section{The reduced phase space}
\label{delzantsec} 
Let $T$ be an $n$\--dimensional torus, a compact, connected, commutative 
$n$\--dimensional real Lie group, with Lie algebra ${\got t}$. 
It follows that the exponential mapping $\op{exp}:{\got t}\to T$ 
is a surjective homomorphism from the additive Lie group 
${\got t}$ onto $T$. Furthermore, ${\got t}_{\Z}:=\op{ker}(\op{exp})$ is a 
discrete subgroup of $({\got t},\, +)$ such that the exponential 
mapping induces an isomorphism from 
${\got t}/{\got t}_{\Z}$ onto $T$, which we also denote by $\op{exp}$. 
Note that ${\got t}_{\Z}$ is defined in terms of the 
group $T$ rather than only the Lie algebra ${\got t}$, but 
the notation ${\got t}_{\Z}$ has the advantage over 
the more precise notation $T_{\Z}$ that it reminds us of the fact  
it is a subgroup of the additive group ${\got t}$. 

Because ${\got t}/{\got t}_{\Z}$ is compact, ${\got t}_{\Z}$ has a $\Z$\--basis 
which at the same time is an $\R$\--basis of ${\got t}$, and 
each $\Z$\--basis of ${\got t}_{\Z}$ is an $\R$\--basis of ${\got t}$.  
Using coordinates with respect to an ordered $\Z$\--basis of 
${\got t}_{\Z}$, we obtain a linear isomorphism from ${\got t}$ 
onto $\R ^n$ which maps ${\got t}_{\Z}$ onto $\Z ^n$, and therefore 
induces an isomorphism from $T$ onto $\R ^n/\Z ^n$. 
For this reason, ${\got t}_{\Z}$ is called the {\em integral lattice} 
in ${\got t}$. However, because we do not have a preferred 
$\Z$\--basis of ${\got t}_{\Z}$, we do not write $T=\R ^n/\Z ^n$. 

Let $\Delta$ be an $n$\--dimensional convex polytope in ${\got t}^*$. 
We denote by $F$ and $V$ the set of all codimension one faces 
and vertices of $\Delta$, respectively. Note that, as a face is 
defined as the set of points of the closed convex set on which a given 
linear functional attains its minimum, see Rockafellar 
\cite[p.162]{r}, every face of $\Delta$ is compact. 
For every $v\in V$, we write 
\[
F_v=\{ f\in F\mid v\in f\} .
\]
$\Delta$ is called a {\em Delzant polytope} if 
it has the following properties, see Guillemin \cite[p. 8]{g}. 
\begin{itemize}
\item[i)] For each $f\in F$ there is an $X_f\in {\got t}$ 
and $\lambda _f\in\R$ such that the hyperplane which contains 
$f$ is equal to the set of all $\xi\in {\got t}^*$ 
such that $\langle X_f,\,\xi\rangle +\lambda _f=0$, and 
$\Delta$ is contained in the set of all $\xi\in {\got t}^*$ 
such that $\langle X_f,\,\xi\rangle +\lambda _f\geq 0$. The vector 
$X_f$ and constant $\lambda _f$ are made unique by requiring 
that they are not 
an integral multiple of another such vector and 
constant, respectively. 
\item[ii)] For every $v\in V$, the $X_f$ with $f\in F_v$ form 
a $\Z$\--basis of the integral lattice ${\got t}_{\Z}$ in ${\got t}$. 
\end{itemize}
It follows that 
\begin{equation}
\Delta =\{\xi\in {\got t}^*\mid 
\langle X_f,\,\xi\rangle +\lambda _f\geq 0
\quad\mbox{\rm  for every}\quad f\in F\} . 
\label{Deltaineq}
\end{equation}
Also, $\# (F_v)=n$ for every $v\in V$, which already 
makes the polytope $\Delta$ quite special. 
In the sequel we assume that $\Delta$ is a given 
Delzant polytope in ${\got t}^*$. 

For any $z\in\C ^F$ and $f\in F$ we write $z(f)=z_f$, 
which we view as the coordinate of the vector $z$ with the index $f$. 
Let $\pi$ 
be the real linear map from $\R ^F$ to ${\got t}$ defined by  
\begin{equation}
\pi (t):=\sum_{f\in F}\, t_f\, X_f,\quad t\in\R ^F.
\label{pidef}
\end{equation}
Because, for any vertex 
$v$, the $X_f$ with $f\in F_v$ form a $\Z$\--basis 
of ${\got t}_{\Z}$ which is also an $\R$\--basis of ${\got t}$, 
we have $\pi (\Z ^F)={\got t}_{\Z}$ and $\pi (\R ^F)={\got t}$. 
It follows that $\pi$ induces a 
surjective homomorphism of Lie groups $\pi '$ from 
the torus $\R ^F/\Z ^F=(\R /\Z)^F$ onto 
${\got t}/{\got t}_{\Z}$, and we have the corresponding 
surjective homomorphism 
$\op{exp}\circ\pi '$ from $\R ^F/\Z ^F$ onto $T$.  

Write ${\got n}:=\op{ker}\pi$, a linear subspace of 
$\R ^F$, and $N=\op{ker}(\op{exp}\circ\pi ')$, 
a compact commutative 
subgroup of the torus $\R ^F/\Z ^F$. 
Actually, $N$ is connected, see Lemma 
\ref{complem} below, and therefore isomorphic 
to ${\got n}/{\got n}_{\Z}$, where 
${\got n}_{\Z}:={\got n}\cap\Z ^F$ is the integral lattice 
in ${\got n}$ of the torus 
$N$. \footnote{We did not find a proof of the connectedness 
of $N$ in \cite{d}, \cite{audin}, or 
\cite{g}.}

On the complex vector space $\C ^F$ of all complex\--valued functions 
on $F$ we have the action of the torus $\R ^F/\Z ^F$, where 
$t\in\R ^F/\Z ^F$ maps $z\in\C ^F$ to the element 
$t\cdot z\in\C ^F$ defined by 
\[
(t\cdot z)_f=\op{e}^{2\pi\scriptop{i}\, t_f}\, z_f,\quad f\in F. 
\]
The infinitesimal action of $Y\in\R ^F=\op{Lie}(\R ^F/\Z ^F)$ 
is given by 
\[
(Y\cdot z)_f=2\pi\op{i}\, Y_f\,z_f,
\]
which is a Hamiltonian vector field defined by the function 
\begin{equation}
z\mapsto\langle Y,\, \mu (z)\rangle 
=\sum_{f\in F} Y_f\, |z_f|^2/2=\sum_{f\in F}\, Y_f\, ({x_f}^2+{y_f}^2)/2,
\label{mu}
\end{equation}
and with respect to the symplectic form 
\begin{equation}
\sigma :=(\op{i}/4\pi)\, \sum_{f\in F}\,\op{d}\! z_f
\wedge\op{d}\!\overline{z}_f
=(1/2\pi )\,\sum_{f\in F}\,\op{d}\! x_f\wedge\op{d}\! y_f, 
\label{sigma}
\end{equation}
if $z_f=x_f+\op{i}y_f$, with $x_f,\, y_f\in\R$. 
Here the factor $1/2\pi$ is introduced in order to avoid 
an integral lattice  $(2\pi\,\Z )^F$ instead of our $\Z ^F$. 

Because the right hand side of (\ref{mu}) depends linearly on $Y$, 
we can view $\mu (z)$ as an element of $(\R ^F)^*\simeq\R ^F$, 
with the coordinates
\begin{equation}
\mu (z)_f=|z_f|^2/2=({x_f}^2+{y_f}^2)/2, \quad f\in F.  
\label{muf}
\end{equation}
In other words, the action of $\R ^F/\Z ^F$ on $\C ^F$ is Hamiltonian, 
with respect to the symplectic form $\sigma$ and with 
momentum mapping $\mu :\C ^F\to (\op{Lie}(\R ^F/\Z ^F))^*$ 
given by (\ref{mu}), or equivalently (\ref{muf}).  

It follows that the subtorus $N$ of $\R ^F/\Z ^F$ acts on 
$\C ^F$ in a Hamiltonian fashion, with momentum mapping 
\begin{equation}
\mu _N:=\iota _{\got n}^*\circ\mu :\C ^F\to {\got n}^*,
\label{muN}
\end{equation}
where $\iota _{\got n}:{\got n}\to\R ^F$ denotes the identity 
viewed as a linear mapping from ${\got n}\subset\R ^F$ 
to $\R ^F$, and its transposed $\iota _{\got n}^*:(\R ^F)^*\to {\got n}^*$ 
is the map which assigns to each linear form on $\R ^F$ its 
restriction to ${\got n}$.  

Write $\lambda _N=\iota _{\got n}^*(\lambda)$, 
where $\lambda$ denotes the element of $(\R ^F)^*\simeq\R ^F$ 
with the coordinates $\lambda _f$, $f\in F$. 
It follows from Guillemin \cite[Th. 1.6 and Th. 1.4]{g}  that 
$\lambda _N$ is a regular value of $\mu _N$, hence the level set  
$Z:={\mu _N}^{-1}(\{ \lambda _N\})$ of $\mu _N$ 
for the level $\lambda _N$ is a smooth 
submanifold of $\C ^F$, and that the action of $N$ on $Z$ 
is proper and free. As a consequence the $N$\--orbit space 
$M=M_{\Delta}:=Z/N$ is a smooth $2n$\--dimensional manifold such that 
the projection $p:Z\to M$ exhibits $Z$ as a principal $N$\--bundle 
over $M$. Moreover, there is a unique symplectic form 
$\sigma _M$ on $M$ such that $p^*\sigma _M={\iota _Z}^*\sigma$, 
where $\iota _Z$ is the identity viewed as a smooth mapping from 
$Z$ to $\C ^F$. 

\begin{remark}
Guillemin \cite{g} used the momentum mapping 
$\mu _N-\lambda _N$ instead of $\mu _N$, such that the reduction 
is taken at the zero level of his momentum mapping. We follow 
Audin \cite[Ch. VI, Sec. 3.1]{audin} in that we use the momentum 
mapping $\mu _N$ for the $N$\--action, which does not depend on 
$\lambda$, and do the reduction at the level $\lambda _N$. 
\label{levelrem}
\end{remark}

The symplectic manifold 
$(M,\,\sigma _M)$ is the {\em Marsden\--Weinstein reduction} 
of the symplectic manifold $(\C ^F,\,\sigma )$ 
for the Hamiltonian $N$\--action at the level $\lambda _N$ of the 
momentum mapping, as defined in Abraham and Marsden 
\cite[Sec. 4.3]{am}. On the $N$\--orbit space $M$, 
we still have the action of the 
torus $(\R ^F/\Z ^F)/N\simeq T$, with momentum mapping 
$\mu _T:M\to {\got t}^*$ determined by 
\begin{equation}
\pi ^*\circ\mu _T\circ p =(\mu -\lambda )|_Z. 
\label{muT}
\end{equation}
The torus $T$ acts effectively on $M$ and $\mu _T(M)=\Delta$, 
see Guillemin \cite[Th. 1.7]{g}. Actually, all these properties 
of the reduction will also follow in a simple way from our description 
in Section \ref{Zsec} of $Z$ in term of the coordinates $z_f$, $f\in F$. 

The symplectic manifold 
$M_{\Delta}$ together with this Hamiltonian $T$\--action 
is called the {\em Delzant space defined by $\Delta$}, 
see Guillemin, \cite[p. 13]{g}. This proves the existence 
part \cite[pp. 328, 329]{d} of Delzant's theory.

\section{The reduced phase space coordinatizations.}
\label{Zsec}
For any $v\in V$, let $\iota _v:={\rho _v}^*: (\R ^{F_v})^*\to (\R ^F)^*$ 
denote the transposed of the restriction projection $\rho _v:\R ^F\to\R ^{F_v}$. 
If in the usual way we identify  $(\R ^{F_v})^*$ and $(\R ^F)^*$  
with $\R ^{F_v}$ and $\R ^F$, respectively, then 
$\iota _v:\R ^{F_v}\to \R ^F$  
is the embedding defined by 
$\iota _v(x)_f=x_f$ if $f\in F_v$ and $\iota _v(x)_{f'}=0$ 
if $f'\in F$, $f'\notin F_v$. Because $\iota _v$ maps 
$\Z ^{F_v}$ into $\Z ^F$ and 
$\iota _v(\R ^{F_v})\cap\Z ^F=\iota _v(\Z ^{F_v})$, it induces an 
embedding of the $n$\--dimensional torus $\R ^{F_v}/\Z ^{F_v}$ 
into $\R ^F/\Z ^F$, which we also denote by $\iota _v$.

\begin{lemma}
With these notations, $\R ^F$, $\Z ^F$, and $\R ^F/\Z ^F$ are the direct 
sum of ${\got n}$ and $\iota _v(\R ^{F_v})$, 
${\got n}\cap\Z ^n$ and $\iota _v(\Z ^{F_v})$, and 
$N$ and $\iota _v(\R ^{F_v}/\Z ^{F_v})$, respectively. 

It follows that $N$ is connected, a torus, with integral lattice 
equal to ${\got n}\cap\Z ^F$. It also follows that 
$\pi\circ\iota _v$ is an isomorphism from 
the torus $\R ^{F_v}/\Z ^{F_v}$ onto the torus $T$. 
\label{complem}
\end{lemma}
\begin{proof}
Let $t\in\R ^F$. Because the 
$X_f$, $f\in F_v$, form an $\R$\--basis 
of ${\got t}$, there exists a unique 
$t^v\in\R ^{F_v}$, such that 
\[
\pi (t)=\sum _{f\in F_v}\, (t^v)_f\, X_f=\pi (\iota _v(t^v)),  
\]
that is, $t-\iota _v(t^v)\in {\got n}$. 
Moreover, because the $X_f$, $f\in F_v$, also form a 
$\Z$\--basis of ${\got t}_{\Z}$, we have that $t^v\in \Z ^{F_v}$, 
and therefore $t-\iota _v(t^v)\in {\got n}\cap\Z ^F$,  
if $t\in \Z ^F$. 
\end{proof}
\begin{lemma}
We have $z\in Z$ if and only if $\mu (z)-\lambda\in 
\pi ^*({\got t}^*)$. More explicitly, if and only if there 
exists a $\xi\in {\got t}^*$ such that 
\begin{equation}
|z_f|^2/2-\lambda _f=\langle X_f,\,\xi\rangle 
\quad\mbox{\rm for every}\quad f\in F. 
\label{xiz}
\end{equation}
When $z\in Z$, the $\xi$ in (\ref{xiz}) is uniquely determined.  

Furthermore, 
$Z=(\mu -\lambda )^{-1}(\pi ^*(\Delta ))$, 
$(\mu -\lambda )(Z)=\pi ^*(\Delta )$, and $Z$ is a 
compact subset of $\C ^F$. 
\label{Zlem}
\end{lemma}
\begin{proof}
The kernel of $\iota _{\got n}^*$ is equal to the space of all linear 
forms on $\R ^F$ which vanish on ${\got n}:=\op{ker}\pi$, and therefore 
$\op{ker}\iota _{\got n}^*$ is equal to the image of 
$\pi ^*:{\got t}^*\to (\R ^F)^*$. Because $\pi$ is surjective, 
$\pi ^*$ is injective, which proves the uniqueness of $\xi$.  

It follows from (\ref{xiz}) that $\langle X_f,\,\xi\rangle +  
\lambda _f\geq 0$ for every $f\in F$, and therefore 
$\xi\in\Delta$ in view of (\ref{Deltaineq}). Conversely, 
if $\xi\in\Delta$, then there exists for every 
$f\in F$ a complex number $z_f$ such that 
$|z_f|^2/2=\langle X_f,\,\xi\rangle +\lambda _f$, 
which means that $z\in Z$ and $(\mu -\lambda )(z)=\pi ^*(\xi )$. 
The set $\pi ^*(\Delta )$ is compact because $\Delta$ is compact 
and $\pi ^*$ is continuous. Because the mapping $\mu -\lambda$ 
is proper, it follows that $Z=(\mu -\lambda )^{-1}(\pi ^*(\Delta ))$ 
is compact.    
\end{proof}
Let $v\in V$. The $X_f$, $f\in F_v$, form an $\R$\--basis of 
${\got t}$, and therefore there exists for each 
$z\in\C ^{F_v}$ a unique $\xi =\mu _v(z)\in {\got t}^*$ such that 
(\ref{xiz}) holds for every $f\in F_v$.  That is, the mapping 
$\mu _v:\C ^{F_v}\to {\got t}^*$ 
is defined by the equations
\begin{equation}
|z_f|^2/2-\lambda _f=\langle X_f,\,\mu _v(z)\rangle ,\quad z\in\C ^{F_v},
\quad f\in F_v. 
\label{muv}
\end{equation}
In other words, $\mu _v$ is defined by the formula 
\begin{equation}
\rho _v\circ\pi ^*\circ\mu _v=\rho _v\circ (\mu -\lambda )\circ\iota _v,
\label{muvformula}
\end{equation}
where $\rho _v$ denotes the restriction projection 
from $\R ^F$ onto $\R ^{F_v}$.  
\begin{lemma}
If we let $T$ act on $\C ^{F_v}$ via $\R ^{F_v}/\Z ^{F_v}$ 
by means of $(t,\, z)\mapsto (\pi\circ\iota _v)^{-1}(t)\cdot z$, 
then $\mu _v:\C ^{F_v}\to {\got t}^*$ 
is a momentum mapping for this Hamiltonian action of 
$T$ on $\C ^{F_v}$, with $\mu _v(0)=v$. Here the symplectic form 
on $\C ^{F_v}$ is equal to 
\begin{equation}
\sigma :=(\op{i}/4\pi)\, \sum_{f\in F_v}\,\op{d}\! z_f
\wedge\op{d}\!\overline{z}_f
=(1/2\pi )\,\sum_{f\in F_v}\,\op{d}\! x_f\wedge\op{d}\! y_f, 
\label{sigmav}
\end{equation}
that is, (\ref{sigma}) with $F$ replaced by $F_v$. 

Let $\rho _v$ denote the restriction projection 
from $\C ^F$ onto $\C ^{F_v}$, and let $U_v$ be the interior 
of the subset $\rho _v(Z)$ of $\C ^{F_v}$. Write 
\begin{equation}
\Delta _v:=\Delta\setminus
\bigcup_{f'\in F\setminus F_v}f'. 
\label{Deltav}
\end{equation}
Then 
$\rho _v(Z)={\mu _v}^{-1}(\Delta )$,
$\mu _v(\rho _v(Z))=\Delta $, 
$U_v={\mu _v}^{-1}(\Delta _v)$, 
and $\mu _v(U_v)=\Delta _v$. 
In particular $\rho _v(Z)$ is a compact subset of 
$\C ^{F_v}$, and $U_v$ is a bounded and connected 
open neighborhood 
of $0$ in $\C ^{F_v}$. 
\label{Zvlem}
\end{lemma}
\begin{proof}
The first statement follows from (\ref{muvformula}), the fact that 
$\rho _v\circ\mu\circ\iota _v$ is a momentum mapping 
for the standard $\R ^{F_v}/\Z ^{F_v}$ action on 
$\C ^{F_v}$, and the fact that a momentum mapping for a 
Hamiltonian action plus a constant is a momentum mapping for the 
same Hamiltonian action. It follows in view of (\ref{muv}) that 
$\langle X_f,\,\mu _v(0)\rangle +\lambda _f=0$ for every 
$f\in F_v$, hence $\mu _v(0)=v$ in view of i) in the definition 
of a Delzant polytope, and the fact that $\{ v\}$ is the intersection 
of all the $f\in F_v$. 
 
It follows from (\ref{muv}), Lemma \ref{Zlem}, that 
$z\in Z$ if and only if 
\begin{equation}
|z_{f}|^2/2=\langle X_{f},\,\mu _v(\rho _v(z))\rangle +\lambda _{f}
\quad\mbox{\rm for every}\quad f\in F, 
\label{Zv}
\end{equation}
where we note that these equations are satisfied by definition for the 
$f\in F_v$. Therefore, if $z\in Z$, then (\ref{Zv}) and 
(\ref{Deltaineq}) imply that $\mu _v(\rho (z))\in\Delta$. 
Conversely, if $\xi\in\Delta$, then it follows from 
Lemma \ref{Zlem} that there exists $z\in Z$ such that 
$\pi ^*(\xi )=\mu (z)-\lambda$, of which the restriction 
to $F_v$ yields $\xi =\mu _v(\rho (z))$. 

If $\xi\in\Delta _v$, $z^v\in\C ^{F_v}$, $\mu _v(z^v)=\xi$, 
then $\langle X_{f'},\,\mu _v(z^v)\rangle +\lambda _{f'}>0$ 
for every $f'\in F\setminus F_v$, which will remain valid 
if we replace $z^v$ by $\widetilde{z}^v$ in a sufficiently small neighborhood 
of $z^v$ in $\C ^{F_v}$.  It follows that we can find $\widetilde{z}\in 
\C ^F$ such that $\rho _v(\widetilde{z})=\widetilde{z}^v$ 
and (\ref{Zv}) holds with $z$ replaced by $\widetilde{z}$. 
That is, $\widetilde{z}\in Z$, and we have proved that 
$z^v\in U_v$. 

Let conversely $z\in U_v\subset\C ^{F_v}$. 
We have in view of (\ref{muv}) that 
\[
|z_f|^2/2=\langle X_f,\,\mu _v(z)-\mu _v(0)\rangle 
=\langle X_f,\,\mu _v(z)-v\rangle
\] 
for every $f\in F_v$. Therefore 
$\mu _v(z)-v$ is multiplied by $c^2$ if we 
replace $z$ by $c\, z$, $c>0$. 
Because $z$ is in the interior of 
$\rho _v(Z)$, we have $c\, z^v\in\rho _v(Z)$, 
hence $\mu _v(c\, z)\in\Delta$ 
for $c>1$, $c$ sufficiently close to $1$.  
On the other hand, if $\xi$ belongs to a 
face of $\Delta$ which is not adjacent to $v$, then 
$v+\tau\, (\xi -v)\notin\Delta$ for any $\tau >1$. 
It follows that $\mu _v(z)$ does not belong to any 
$f'\in F\setminus F_v$, that is, $\mu _v(z)\in\Delta _v$.  
\end{proof}

The equation (\ref{Zv}) can be written in the form 
$|z_f|=r_f(\mu _v(\rho _v(z)))$, where, for each $f\in F$,  
the function $r_f:\Delta\to\R _{\geq 0}$ is defined by 
\begin{equation}
r_f(\xi ):=(2(\langle X_f,\,\xi\rangle +\lambda _f))^{1/2},\quad 
f\in F,\;\xi\in\Delta . 
\label{rv}
\end{equation}
We now view the equations (\ref{Zv}) for $z\in Z$ as equations 
for the coordinates $z_{f'}$, $f'\in F\setminus F_v$, 
with the $z_f$, $f\in F_v$ as parameters, where the latter 
constitute the vector $z^v=\rho _v(z)$. 
If $z^v\in U_v$, then for each $f'\in F\setminus F_v$ the coordinate 
$z_{f'}$ lies on the circle  about the origin 
with {\em strictly positive} radius $r_{f'}(\mu _v(z^v)$. 
Because Lemma \ref{complem} implies that 
the homomorphism which assigns to each element 
of $N$ its projection to $\R ^{F\setminus F_v}/\Z ^{F\setminus F_v}$ 
is an isomorphism, and the latter torus is the group 
of the coordinatewise rotations of the $z_{f'}$, 
$f'\in F\setminus F_v$, this leads to the following conclusions. 
\begin{proposition}
Let $v$ be a vertex of $\Delta$. 
The open subset $Z_v:={\rho _v}^{-1}(U_v)\cap Z$ of $Z$ is 
a connected smooth submanfold of $\C ^F$ of real dimension 
$2n+(d-n)$, where $d=\#(F)$ and $d-n=\op{dim}N$. 
The action of the torus $N$ on $Z_v$ is free, and the 
projection $\rho _v:Z_v\to U_v$ exhibits $Z_v$ as a 
principal $N$\--bundle over $U_v$. It follows that 
we have a reduced phase space $M_v:=Z_v/N$, 
which is a connected smooth symplectic $2n$\--dimensional manifold, which 
carries an effective Hamiltonian $T$\--action with momentum 
mapping as in (\ref{muT}), with $Z$ replaced by $Z_v$. 

There is a unique global section $s_v:U_v\to Z_v$ of 
$\rho _v:Z_v\to U_v$ such that 
$s_v(z)_{f'}\in\R _{>0}$ for every $z\in U_v$ and $f'\in F\setminus F_v$. 
Actually, $s_v(z)_{f'}=r_{f'}(\mu _v(z))$ when $z\in U_v$ and 
$f'\in F\setminus F_v$, 
and therefore the section $s_v$ is smooth.  
If $p_v:Z_v\to M_v=Z_v/N$ 
denotes the canonical projection, then 
$\psi _v:=p_v\circ s_v$ is a $T$\--equivariant symplectomorphism 
from $U_v$ onto $M_v$, where $T$ acts on 
$U_v$ via  $\R ^{F_v}/\Z ^{F_v}$, as in Lemma \ref{Zvlem}. 
\label{fibvprop}
\end{proposition}
\begin{remark}
When $z$ belongs to the closure $\rho _v(Z)=\overline{U_v}$ of $U_v$ 
in $\C ^{F_v}$, see Lemma \ref{Zvlem}, we can define $s_v(z)\in\C ^F$ 
by $s_v(z)_f=z_f$ when $f\in F_v$ and 
$s_v(z)_{f'}=r_{f'}(\mu _v(z))$ when $f'\in F\setminus F_v$. 
This defines a continuous extension $s_v:\overline{U_v}\to\C ^F$ 
of the mapping $s_v:U_v\to Z$. Therefore 
$s_v(\overline{U_v})\subset Z$, and $\psi _v :=p\circ s_v:
\overline{U_v}\to M$ is a continuous extension of the 
diffeomorphism $\psi _v:U_v\to M_v$. 

The continuous mapping $\psi _v:\overline{U_v}\to M$ is 
surjective, but the restriction of it to the boundary 
$\partial U_v:=\overline{U_v}\setminus U_v$ of $U_v$ in 
$\C ^{F_v}$ is not injective. If $z^v\in\partial U_v$, 
then the set $G$ of all $f'\in F\setminus F_v$ such that 
$s_v(z^v)_{f'}=0$, or equivalently 
$\mu _T(\psi _v(z^v))\in f'$,  is not empty. 
The fiber of $\psi _v$ over $\psi _v(z^v)$ is equal to 
the set of all $t^v\cdot z^v$, where the $t^v\in\R ^{F_v}/\Z ^{F_v}$ 
are of the form 
\[
t^v_f=\,-\sum_{g\in G}\, (v)^f_g\, t_g,\quad f\in F_v,
\]
where $t_g\in\R /\Z$. It follows that each fiber is an 
orbit of some subtorus of $\R ^{F_v}/\Z ^{F_v}$ acting on $\C ^{F_v}$.  
\label{Uvclosurerem}
\end{remark}

Recall the definition (\ref{Deltav}) of the open 
subset $\Delta _v$ of the Delzant polytope $\Delta$. 
Because the union over all vertices $v$ of the 
$\Delta _v$ is equal to $\Delta$, we have the following corollary.  
\begin{corollary}
The sets $Z_v$, $v\in V$, form a covering of $Z$. As a consequence, 
$Z$ is a smooth submanifold of $\C ^F$ of real dimension $n+d$. 
The action of the torus $N$ on $Z$ is free, and 
we have a reduced phase space $M:=Z/N$, 
which is a compact and connected 
smooth $2n$\--dimensional symplectic manifold, which 
carries an effective Hamiltonian $T$\--action with momentum 
mapping $\mu _T:M\to T$ as in (\ref{muT}). The sets $M_v$, $v\in V$, 
form an open covering of $M$ and the 
$\varphi _v:=(\psi _v)^{-1}:M_v\to U_v$ form an atlas 
of $T$\--equivariant symplectic coordinatizations 
of the Hamiltonian $T$\--space $M$. 
For each $v\in V$, we have $M_v={\mu _T}^{-1}(\Delta _v)$, 
and $\mu _T|_{M_v}=\mu _v\circ\varphi _v$. 
\label{fibvcor}
\end{corollary}
For a characterization of $M_v$ in terms of the orbit type 
stratification in $M$ for the $T$\--action, see Corollary 
\ref{stratacor}, which also implies that $M_v$ is an open cell in $M$.  
\begin{corollary}
For every $f\in F$ the set ${\mu _T}^{-1}(f)$ is a real codimension
two smooth compact connected smooth symplectic submanifold of $M$. 

For each $v\in V$, the set $M_v$ is dense in $M$, and the 
diffeomorphism $\psi _v:U_v\to M_v$ is maximal among all 
diffeomorphisms from open subsets of $\C ^{F_v}$ onto 
open subsets of $M$.
\label{fibvcorcor} 
\end{corollary}
\begin{proof}
If $f\in F$, then for each $v\in V$ we have that 
\begin{equation}
{\mu _v}^{-1}(f)=\{ z\in U_v\mid z_f=0\}  
\label{mu-1f}
\end{equation}
if $v\in f$, that is, $f\in\Delta _v$. This follows from 
(\ref{muv}) and i) in the description of $\Delta$ in the 
beginning of Section \ref{delzantsec}. On the other hand, 
${\mu _v}^{-1}(f)=\emptyset$ if $f\notin\Delta _v$. 
Because ${\mu _T}^{-1}(f)\cap M_v=\psi _v({\mu _v}^{-1}(f))$,  
and the $M_v$, $v\in V$, form an open covering of $M$, 
this proves the first statement. The second statement 
follows from the first one, because the complement of $M_v$ 
in $M$ is equal to the union of the sets ${\mu _T}^{-1}(f')$ 
with $f'\in F\setminus F_v$.   
\end{proof}
\begin{remark}
It follows from the proof of Corollary 
\ref{fibvcorcor}, that ${\mu _T}^{-1}(f)$ 
is a connected component of the fixed point set in $M$ of the 
of the circle subgroup $\op{exp}(\R\, X_f)$ of $T$. 

Actually,  ${\mu _T}^{-1}(f)$ is 
a Delzant space for the action of the $(n-1)$\--dimensional torus
\[
T/\op{exp}(\R\, X_f), 
\]
with Delzant polytope $P\subset ({\got t}/(\R\, X_f))^*$ 
such that the image of $P$ in ${\got t}^*$ under the 
embedding  $({\got t}/(\R\, X_f))^*\to {\got t}^*$ is equal to 
a translate of $f$. 

In a similar way, if 
$g$ is a $k$\--dimensional face of $\Delta$, 
then ${\mu _T}^{-1}(g)$ 
is a $2k$\--dimensional Delzant space for the quotient of $T$ by the 
subtorus of $T$ which acts trivially on ${\mu _T}^{-1}(g)$. 
\label{fibvcorcorrem}
\end{remark}
\begin{remark}
Let, for each $f\in F$, $c_f\in\op{H}^2(M,\,\Z )\subset\op{H}^2(M,\,\R )$ 
denote the Poincar\'e dual of the codimension two 
Delzant subspace ${\mu _T}^{-1}(f)$ of $M$, see 
Remark \ref{fibvcorcorrem}. Then, with 
our normalization  of the symplectic form (\ref{sigma}), 
the de Rham cohomology class $[\sigma _M]$ of the 
symplectic form $\sigma _M$ of the Delzant space is 
equal to 
\begin{equation}
[\sigma _M]=\sum_{f\in F}\,\lambda _f\, c_f, 
\label{cohom}
\end{equation}
see Guillemin \cite[Thm. 6.3]{gpaper}. In particular 
$[\sigma _M]\in\op{H}^2(M,\,\Z )$, and therefore 
$[\sigma _M]$ is equal to the Chern class of a complex 
line bundle over $M$, if all the coefficients 
$\lambda _f$, $f\in F$, are integers. 

If $\Delta$ is a simplex, when $M$ is isomorphic to the 
$n$\--dimensional complex projective space, then the 
${\mu _T}^{-1}(f)$, $f\in F$, are complex projective hyperplanes, see 
Subsection \ref{projsubsec}, which are all homologous to each other. 
It follows that in this case $[\sigma _M]=\gamma\, c$, where 
$c$ is the Poincar\'e dual of a complex projective hyperplane and 
$\gamma$ is equal to the sum of all the coefficients 
$\lambda _f$, $f\in F$. 
\label{cohomrem}
\end{remark}
\begin{remark}
Let $\iota :T\to \R ^n/\Z ^n$ be an isomorphism of tori, 
which allows us to let $t\in T$ act on $\C ^n$ via 
$\R ^n /\Z ^n$ by means of 
\[
(t\cdot z)_j=\op{e}^{2\pi\scriptop{i}\iota (t)_j}\, z_j,\quad 
1\leq j\leq n.
\]
Let $U$ be a connected $T$\--invariant open neighborhood 
of $0$ in $\C ^n$, provided with the symplectic form (\ref{sigma}) with 
$F$ replaced by $\{ 1,\,\ldots ,\, n\}$. Let 
$\psi :U\to M$ be a $T$\--equivariant symplectomorphism 
from $U$ onto an open subset $\psi (U)$ of $M$. 
Because $0$ is the unique fixed point for the $T$\--action 
in $U$, and the fixed points for the $T$\--action in $M$ are 
the pre\--images under $\mu _T$ of the vertices 
of $\Delta$, there is a unique $v\in V$ such that 
$\mu _T(\psi (0))=v$. Let $I_v:\C ^{F_v}\to\C ^n$ denote the 
complex linear extension of the tangent map of the 
torus isomorphism $\iota\circ (\pi\circ\iota _v)$. 
In terms of the notation of 
Lemma \ref{Zvlem} and Proposition \ref{fibvprop}, 
we have that $U\subset I_v(U_v)$ 
and $\psi _v=\psi\circ I_v$ on ${I_v}^{-1}(U)$, which leads to 
an identification of $\psi$ with the restriction of 
$\psi _v$ to the connected open subset ${I_v}^{-1}(U)$ 
of $U_v$, via the isomorphism ${I_v}^{-1}$. 

The $\psi$'s, with $U$ equal to a ball in $\C ^n$ 
centered at the origin, are the equivariant symplectic ball 
embeddings in Pelayo \cite{pelayo}, and the  
second statement in Corollary \ref{fibvcorcor} shows that 
the diffeomorphisms $\psi _v$ are the maximal extensions 
of these equivariant symplectic ball embeddings. 
\label{pelayorem}
\end{remark} 
\section{The coordinate transformations}
\label{coordtransfsec}
Let $v,\, w\in V$. Then 
\begin{eqnarray}
U_{v,\, w}&:=&\varphi _v(M_v\cap M_w)=U_v\cap    
{ \psi _v}^{-1}\circ\psi _w(U_w)
\nonumber\\
&=&\{ z^v\in U_v\mid (z^v)_f\neq 0
\quad\mbox{\rm for every}\quad f\in F_v\setminus F_w\} .  
\label{Uvw}
\end{eqnarray}
In this section we will give an explicit formula for the 
coordinate transformations 
\[
\varphi _w\circ {\varphi _v}^{-1}={\psi _w}^{-1}\circ\psi _v: 
U_{v,\, w}\to U_{w,\, v}, 
\]
which then leads to a description 
of the Delzant space $M$ as obtained by gluing together 
the subsets $U_v$ with the coordinate transformations as 
the gluing maps.  

Let $f\in F$. Because the $X_g$, $g\in F_w$, form a $\Z$\--basis 
of ${\got t}_{\Z}$, and $X_f\in {\got t}_{\Z}$, there exist unique integers 
$(w)_f^g$, $g\in F_w$,  such that 
\begin{equation}
X_f=\sum_{g\in F_w}\, (w)_f^g\, X_g.  
\label{mfg}
\end{equation}
Note that if $f\in F_w$, then $(w)_f^g=1$ when $g=f$ and 
$(w)_f^g=0$ otherwise. For the following lemma recall
that $r_g$ is defined by expression (\ref{rv}).
\begin{lemma}
Let $v,\, w\in V$, $z^v\in U_{v,\, w}$. Then  
$z^w:=\varphi _w\circ {\varphi _v}^{-1}(z^v)\in U_w\subset\C ^{F_w}$ 
is given by 
\begin{equation}
z^w_g=\prod_{f\in F_v}(z^v_f)^{(w)_f^g}\;
/\prod_{f\in F_v\setminus F_w}\, |z^v_f|^{(w)_f^g}
\label{inFv}
\end{equation}
if $g\in F_w\cap F_v$, and 
\begin{equation}
z^w_g=\prod_{f\in F_v}(z^v_f)^{(w)_f^g}\; 
r_g(\mu _v(z^v))
/\prod_{f\in F_v\setminus F_w}\, |z^v_f|^{(w)_f^g}
\label{notinFv}
\end{equation}
if $g\in F_w\setminus F_v$. 
\label{coordtranslem}
\end{lemma}
\begin{proof}
The element $z^w\in U_w$ is determined by the condition that  
$s_w(z^w)$ belongs to the $N$\--orbit of $s_v(z^v)$. That is, 
\[
s_w(z^w)_f=\op{e}^{\fop{i}t_f}\, s_v(z^v)_f\quad
\mbox{\rm for every}\quad f\in F
\] 
for some $t\in\R ^F$ such that  
\begin{equation}
\sum_{f\in F}\, t_f\, X_f=0.
\label{tn}
\end{equation}
It follows from (\ref{tn}), (\ref{mfg}) and the linear independence 
of the $X_g$, $g\in F_w$, that $t\in {\got n}$ if and only if 
\begin{equation}
t_g=\, -\sum_{f\in F\setminus F_w}\, (w)_f^g\, t_f\quad 
\mbox{\rm for every}\quad g\in F_w.
\label{tg}
\end{equation}

Note that $\mu _v(z^v)=\mu _T(m)=\mu _w(z^w)$, where 
$m=\psi _v(z^v)=\psi_w(z^w)$. 
It follows from the definition of the sections $s_v$ and $s_w$, 
see Proposition \ref{fibvprop}, that 
\begin{itemize}
\item[i)] 
$s_v(z^v)_f=z^v_f$ and $s_w(z^w)_f=z^w_f$ if $f\in F_v\cap F_w$,
\item[ii)] 
$s_v(z^v)_f=z^v_f$ and $s_w(z^w)_f=r_f(\mu _w(z^w))=r_f(\mu _v(z^v))$ 
if $f\in F_v\setminus F_w$, 
\item[iii)]
$s_v(z^v)_f=r_f(\mu _v(z^v))$ and $s_w(z^w)_f=z^w_f$ 
if $f\in F_w\setminus F_v$, and 
\item[iv)] 
$s_v(z^v)_f=r_f(\mu _v(z^v))=r_f(\mu _w(z^w))=s_w(z^w)_f$ 
if $f\in F\setminus (F_v\cup F_w)$. 
\end{itemize}
It follows from ii) and iv) that $t_f=\, -\op{arg}z^v_f$ 
and $t_f=0$ modulo $2\pi$ if 
$f\in F_v\setminus F_w$ and $f\in F\setminus 
(F_v\cup F_w)$, respectively. Then (\ref{tg}) implies that, 
modulo $2\pi$, 
\[
t_g=\sum_{f\in F_v\setminus F_w}\, (w)_f^g\, \op{arg}z^v_f
\quad\mbox{\rm for every}\quad g\in F_w.
\]

It now follows from i) and iii) that if $g\in F_w$, then 
$z^w_g=s_w(z^w)_g=\op{e}^{\fop{i}\, t_g}\, s_v(z^v)_g$ 
is equal to 
\[
\op{e}^{\fop{i}\, t_g}\, z^v_g
=\prod_{f\in F_v}(z^v_f)^{(w)_f^g}\;
/\prod_{f\in F_v\setminus F_w}\, |z^v_f|^{(w)_f^g}
\]
if $g\in F_v$, and equal to 
\[
\op{e}^{\fop{i}\, t_g}\, |z^v_g|
=\prod_{f\in F_v}(z^v_f)^{(w)_f^g}\;
|z^v_g|/\prod_{f\in F_v\setminus F_w}\, |z^v_f|^{(w)_f^g}
\]
if $g\notin F_v$, respectively. 
Here we have used that if $g\in F_w$, then $(w)_f^g=1$ if 
$f=g$ and  $(w)_f^g=0$ if $f\in F_w$, $f\neq g$. 
Because $|z^v_g|=r_g(\mu _v(z^v))$ if $g\notin F_v$, see 
(\ref{Zv}) and (\ref{rv}), this completes the proof of the lemma. 
\end{proof}
\begin{remark}
Note that $z^v\in U_{v,\, w}$ means that 
$z^v\in U_v$ and $z^v_f\neq 0$ if $f\in F_v\setminus F_w$. 
Furthermore, $z^v\in U_v$ implies that if $g\notin F_v$, then 
$\mu _v(z^v)\notin g$,  and therefore $r_g$ is smooth 
on a neighborhood of $\mu _v(z^v)$. Finally, note that 
if $g\in F_w$ and $f\in F_v\cap F_w$, then 
$(w)_f^g\in \{ 0,\, 1\}$, and therefore each of the factors 
in the right hand sides of (\ref{inFv}) and (\ref{notinFv}) 
is smooth on $U_{v,\, w}$. 
\label{smoothrem}
\end{remark}
\begin{remark}
In (\ref{inFv}) and (\ref{notinFv}) only the integers 
$(w)_f^g$ appear with $f\in F_v$ and $g\in F_w$. 
Let $(w\, v)$ denote the matrix $(w)_f^g$, 
where $f\in F_v$ and $g\in F_w$. Then $(w\, v)$ 
is invertible, with inverse equal to 
the integral matrix $(v\, w)$. 
These integral matrices also satisfy the 
cocycle condition that $(w\, v)\, (v\, u)=(w\, u)$, 
if $u,\, v,\, w\in V$. These properties follow 
from the fact that (\ref{mfg}) shows that $(w\, v)$ 
is the matrix which maps the 
$\Z$\--basis $X_g$, $g\in F_w$, onto the $\Z$\--basis 
$X_f$, $f\in F_v$, of ${\got t}_{\Z}$. It is no surprise that 
these {\em base changes} enter in the formulas which relate the models
in the vector spaces $\C ^{F_v}$ for the different choices of $v\in V$. 
\label{(w)rem}
\end{remark}
\begin{corollary}
Let, for each $v\in V$, the mapping $\mu _v:\C ^{F_v}\to {\got t}^*$ 
be defined by  (\ref{muv}), which is a momentum mapping 
for a Hamiltonian $T$\--action via $\R ^{F_v}/\Z ^{F_v}$ 
on the symplectic vector space $\C ^{F_v}$ 
as in Lemma \ref{Zvlem}. Define $U_v:={\mu _v}^{-1}(\Delta _v)$. 
If also $w\in V$, define $U_{v,\, w}$ as the right hand side 
of (\ref{Uvw}), and, if $z^v\in U_{v,\, w}$, define 
$\varphi _{w,\, v}(z^v):=z^w$, where $z^w\in\C ^{F_w}$ is 
given by (\ref{inFv}) and (\ref{notinFv}). 

Then $\varphi _{w,\, v}$ is a $T$\--equivariant symplectomorphism 
from $U_{v,\, w}$ onto $U_{w,\, v}$ such that 
$\mu _w=\mu _v\circ\varphi _{v,\, w}$ on $U_{w,\, v}$. 
The $\varphi _{w,\, v}$ satisfy the cocycle 
condition $\varphi _{w,\, v}\circ\varphi _{v,\, u}=\varphi _{w,\, u}$ 
where the left hand side is defined. 
Glueing together the Hamiltonian $T$\--spaces 
$U_v$, $v\in V$, with the momentum maps $\mu _v$, 
by means of the gluing 
maps $\varphi _{w,\, v}$, $v,\, w\in V$, we obtain a 
compact connected smooth symplectic manifold $\widetilde{M}$ 
with an effective Hamiltonian $T$\--action with a 
common momentum map $\widetilde{\mu} :\widetilde{M}\to T$ such that 
$\widetilde{\mu} (\widetilde{M})=\Delta$. In other words, 
$\widetilde{M}$ is a Delzant 
space for the Delzant polytope $\Delta$.  
\label{coordtransfcor}
\end{corollary}
The Delzant space $\widetilde{M}$ is obviously isomorphic to the 
Delzant space $M=\mu ^{-1}(\{ \lambda\})/N$ introduced in 
Section \ref{delzantsec}, and actually the isomorphism 
is used in the proof that $\widetilde{M}$ is a Delzant space 
for the Delzant polytope $\Delta$. The only purpose of 
Corollary \ref{coordtransfcor} is to exhibit the 
Delzant space as obtained from gluing together 
the $U_v$, $v\in V$, by means of the gluing maps 
$\varphi _{v,\, w}$, $v,\, w\in V$.  
\section{The toric variety}
\label{torvarsec}
Let $\T :=\{ z\in\C\mid |z|=1\}$ denote the unit circle 
in the complex plane. The mapping $t\mapsto u$ 
where $u_f=\op{e}^{2\pi\scriptop{i}t_f}$ for every $f\in F$ 
is an isomorphism from the torus $\R ^F/\Z ^F$ onto $\T ^F$, 
where $\T ^F$ acts on $\C ^F$ by means of coordinatewise 
multiplication and $\R ^F/\Z ^F$ acted on $\C ^F$ via the 
isomorphism from $\R ^F/\Z ^F$ onto $\T ^F$. The complexification 
$\T _{\C}$ of the compact Lie group $\T$ is the multiplicative group 
$\C ^{\times}$ of all nonzero complex numbers, and  
the complexification of $\T ^F$ is equal to $\T _{\C}^F:=
(\T _{\C})^F=(\C ^{\times})^F$, which also acts on 
$\C ^F$ by means of coordinatewise multiplication. 

The complexification $N_{\C}$ of $N$ is the subgroup 
$\op{exp}({\got n}_{\C})$ of $U_{\C}^F$, where 
$\got{n}_{\C}:=\got{n}\oplus\op{i}{\got n}\subset\C ^F$ 
denotes the complexification of ${\got n}$, viewed as a 
complex linear subspace, a complex Lie subalgebra, of the 
Lie algebra $\C ^F$ of $\T _{\C}^F$. In view of 
(\ref{tg}), we have, for every $v\in V$, that $N_{\C}$ is equal 
to the set of all $t\in \T _{\C}^F$ such that 
\begin{equation}
t_g=\prod_{f\in F\setminus F_v}\, {t_f}^{-(v)_f^g},
\quad g\in F_v. 
\label{NC}
\end{equation}
This implies that $N_{\C}$ is a closed subgroup of $\T _{\C}^F$ 
isomorphic to $\T _{\C}^{F\setminus F_v}$, 
and therefore $N_{\C}$ is a reductive complex algebraic group. 

If we define 
\begin{equation}
\C ^F_v:=\{ z\in\C ^F\mid z_f\neq 0\quad
\mbox{\rm for every}\quad f\in F\setminus F_v\} ,
\label{CFv}
\end{equation}
then it follows from (\ref{NC}) that the action of 
$N_{\C}$ on $\C ^F_v$ is free and proper. 
It follows that the action of $N_{\C}$ on 
\begin{equation}
\C ^F_{\Delta}=\bigcup _{v\in V}\,\C ^F_v
\label{CFDelta}
\end{equation}
is free and proper, and therefore the $N_{\C}$\--orbit space 
\begin{equation}
M^{\scriptop{toric}}:=\C ^F_{\Delta}/N_{\C}
\label{Mtoric}
\end{equation}
has a unique structure of a complex analytic manifold 
of complex dimension $n$ such that the canonical 
projection from $\C ^F_{\Delta}$ onto $M^{\scriptop{toric}}$ 
exhibits $\C ^F_{\Delta}$ as a principal $N_{\C}$\--bundle 
over $M^{\scriptop{toric}}$. On $M^{\scriptop{toric}}$ we still 
have the complex analytic action of the complex Lie group group 
$\T _{\C}^F/N_{\C}$, which is isomorphic to the complexification 
$T_{\C}$ of our real torus $T$ induced by the 
projection $\pi$. The complex analytic manifold 
$M^{\scriptop{toric}}$ together with the complex analytic 
action of $T_{\C}$ on it is the {\em toric variety defined by 
the polytope $\Delta$} in the title of this section. 

If $v\in V$ and $z\in\C ^F_v$, then it follows from (\ref{NC}) 
that there is a unique $t\in N_{\C}$ such that 
$t_f=z_f$ for every $f\in F\setminus F_v$, or 
in other words, $z=t\cdot \zeta$, where 
$\zeta\in\C ^F$ is such that $\zeta _f=1$ for every 
$f\in F\setminus F_v$. Let $S_v:\C ^{F_v}\to 
\C ^F_v$ be defined by $S_v(z^v)_f=z^v_f$ when $f\in F$ 
and $S_v(z^v)=1$ when $f\in F\setminus F_v$, 
as in Audin \cite[p. 159]{audin}. 
If $P_v:\C ^F_v\to \C ^F_v/N_{\C}$ denotes the canonical 
projection from $\C ^F_v$ onto the open subset 
$M^{\scriptop{toric}}_v:= \C ^F_v/N_{\C}$ of 
$M^{\scriptop{toric}}$, then $\Psi _v:=P_v\circ S_v$ 
is a complex analytic diffeomorphism from 
$\C ^{F_v}$ onto $M^{\scriptop{toric}}_v$. 
It is $T_{\C}$\--equivariant if we let $T_{\C}$ act on 
$\C ^{F_v}$ via $\T ^{F_v}_{\C}$ as in Lemma \ref{Zvlem}. 
We use the diffeomorphism $\Phi _v:={\Psi _v}^{-1}$ 
from $M^{\scriptop{toric}}_v$ onto $\C ^{F_v}$ as a 
coordinatization of the open subset $M^{\scriptop{toric}}_v$ 
of $M^{\scriptop{toric}}$. 

If $v,\, w\in V$, then 
\begin{eqnarray}
U_{v,\, w}^{\scriptop{toric}}
&:=&\Phi _v(M^{\scriptop{toric}}_v\cap M^{\scriptop{toric}}_w)
=\C ^{F_v}\cap{\Psi _v}^{-1}\circ\Psi _w(\C ^{F_w})
\nonumber\\
&=&\{ z^v\in\C ^{F_v}\mid (z^v)_f\neq 0
\quad\mbox{\rm for every}\quad f\in F_v\setminus F_w\} .  
\label{Uvwtoric}
\end{eqnarray}
Moreover, with a similar argument as for Lemma \ref{coordtranslem}, 
actually much simpler, we have that for every 
$z^v\in U^{\scriptop{toric}}_{v,\, w}$ the element  
$z^w:=\Phi _w\circ {\Phi _v}^{-1}(z^v)\in\C ^{F_w}$ is given by 
\begin{equation}
z^w_g=\prod_{f\in F_v}\, (z^v_f)^{(w)^g_f},\quad g\in F_w.  
\label{coordtranstoric}
\end{equation}
In this way the coordinate transformation 
$\Phi _w\circ {\Phi _v}^{-1}$ is a Laurent monomial mapping, 
much simpler than the coordinate transformation 
(\ref{inFv}), (\ref{notinFv}). 
It follows that the toric variety $M^{\scriptop{toric}}$ 
can be alternatively described as obtained by 
gluing the $n$\--dimensional complex vector spaces 
$\C ^{F_v}$, $v\in V$, together, with the 
maps (\ref{coordtranstoric}) as the gluing maps. 
This is the kind of toric varieties as introduced by 
Demazure \cite[Sec. 4]{demazure}. 

For later use we mention the following 
observation of Danilov \cite[Th. 9.1]{danilov}, 
which is also of interest in itself. 
\begin{lemma}
$M^{\scriptop{toric}}$ is simply connected. 
\label{danilovlem}
\end{lemma}
\begin{proof}
Let $w\in V$. It follows from (\ref{Uvwtoric}), for all $v\in V$, that
the complement of $M^{\scriptop{toric}}_w$ in 
$M_{\scriptop{toric}}$ is equal to the union of finitely 
closed complex analytic submanifolds of complex codimension 
one, whereas $M^{\scriptop{toric}}_w$ is contractible because 
it is diffeomorphic to the complex vector space $\C ^{F_w}$. 
Because complex codimension one is real codimension two, 
any loop in $M^{\scriptop{toric}}$ with base point in 
$M^{\scriptop{toric}}_w$ can be slightly deformed to  
such a loop which avoids the complement  of 
$M^{\scriptop{toric}}_w$ in 
$M_{\scriptop{toric}}$, that is, which is contained in 
$M^{\scriptop{toric}}_w$, after which it can be contracted 
within $M^{\scriptop{toric}}_w$ to the base point in 
$M^{\scriptop{toric}}_w$.  
\end{proof}
Recall the definition in Section \ref{delzantsec} of the reduced phase space 
$M=Z/N$. 
\begin{theorem}
The identity mapping from $Z$ into $\C ^F_{\Delta}$, 
followed by the canonical projection $P$ from 
$\C ^F_{\Delta}$ to $M^{\scriptop{toric}}=\C ^F_{\Delta}/N_{\C}$, 
induces a $T$\--equivariant diffeomorphism $\varpi$ from 
$M=Z/N$ onto $M^{\scriptop{toric}}$. It follows that 
each $N_{\C}$\--orbit in $\C ^F_{\Delta}$ intersects 
$Z$ in an $N$\--orbit in $Z$. 
\label{idthm}
\end{theorem}
\begin{proof}
Because $N$ is a closed Lie subgroup of $N_{\C}$, 
we have that the  mapping $P:Z\to \C ^F_{\Delta}/N_{\C}$ 
induces a mapping $\varpi :Z/N\to \C ^F_{\Delta}/N_{\C}$, 
which moreover is smooth. 

If $v\in V$, $z\in Z_v$, then it 
follows from (\ref{NC}) that the $t_f$, $f\in F\setminus F_v$, 
of an element $t\in N_{\C}$ can take arbitrary values, 
and therefore the $|z_f|$, $f\in F\setminus F_v$ 
can be moved arbitrarily by means of infinitesimal 
$N_{\C}$\--actions. Because $Z$ is defined by prescribing the 
$|z_f|$, $f\in F\setminus F_v$, as a smooth 
function of the $z_f$, $f\in F_v$, and the $Z_v$, $v\in V$, 
form an open covering of $Z$, this shows that 
at each point of $Z$ the $N_{\C}$\--orbit is transversal 
to $Z$, which implies that $\varpi$ is a submersion. 

It follows that $\varpi (M)$ is an open subset of 
$M^{\scriptop{toric}}$. Because $M$ is compact and 
$\varpi$ is continuous, $\varpi (M)$ is compact, and 
therefore a closed subset of $M^{\scriptop{toric}}$. 
Because $M^{\scriptop{toric}}$ is connected, the conclusion is 
that $\varpi (M)=M^{\scriptop{toric}}$, that is, $\varpi$ 
is surjective. 

Because $\varpi$ is a surjective submersion, $\op{dim}_{\R}M=2n=
\op{dim}_{\R}M^{\scriptop{toric}}$, and $M$ is 
connected, we conclude that $\varpi$ is a covering 
map. Because $M^{\scriptop{toric}}$ is simply connected, 
see Lemma \ref{danilovlem}, we conclude that 
$\varpi$ is injective, that is, $\varpi$ is a diffeomorphism. 
\end{proof}
\begin{remark}
Theorem \ref{idthm} is the last statement 
in Delzant \cite{d}, with no further details of the proof. 
Audin \cite[Prop. 3.1.1]{audin} gave a proof using 
gradient flows, whereas the injectivity has been proved 
in \cite[Sec. A1.2]{g} using the principle that 
the gradient of a strictly convex function defines 
an injective mapping. 
\label{idthmrem}
\end{remark}
Note that in the definition of the toric variety 
$M^{\scriptop{toric}}$, the real numbers $\lambda _f$, $f\in F$, 
did not enter, whereas these numbers certainly enter in the 
definition of $M$, the symplectic form on $M$, and the 
diffeomorphism $\varpi :M\to M^{\scriptop{toric}}$. 
Therefore the symplectic form $\sigma ^{\scriptop{toric}}_{\lambda} :=
(\varpi ^{-1})^*(\sigma )$ on $M^{\scriptop{toric}}$ 
on $M^{\scriptop{toric}}$ will depend on the choice of $\lambda\in\R ^F$. 
On the symplectic manifold 
$(M^{\scriptop{toric}},\,\sigma ^{\scriptop{toric}}_{\lambda})$, 
the action of the maximal compact subgroup $T$ of $T_{\C}$ 
is Hamiltonian, with momentum mapping equal to 
\begin{equation}
\mu ^{\scriptop{toric}}_{\lambda} :=\mu\circ\varpi ^{-1}:
M^{\scriptop{toric}}\to {\got t}^*,
\label{mutor}
\end{equation}
where $\mu ^{\scriptop{toric}}_{\lambda}(M^{\scriptop{toric}})=\Delta$, 
where we note that $\Delta$ in (\ref{Deltaineq}) depends on $\lambda$. 

In the following lemma we compare the 
reduced phase space coordinatizations with the 
toric variety coordinatizations. 
\begin{lemma}
Let $v\in V$. Then $M_v^{\scriptop{toric}}
=\varpi (M_v)$, and 
\begin{equation}
\theta _v:= {\Psi _v}^{-1}\circ\varpi\circ\psi _v  
\label{thetav}
\end{equation} 
is a $\T ^{F_v}$\--equivariant diffeomorphism from $U_v$ 
onto $\C ^{F_v}$. 

For each $z^v\in U_v$, the element 
$\zeta ^v:=\theta _v(z^v)$ is given in terms of $z^v$ by 
\begin{equation}
\zeta ^v_f=z^v_f\,\prod_{f'\in F\setminus F_v}\, r_{f'}(\mu _v(z^v))^{(v)^f_{f'}},
\quad f\in F_v,
\label{zzeta}
\end{equation}
where the functions $r_{f'}:\Delta\to\R _{\geq 0}$ are given by (\ref{rv}). 
We have 
\begin{equation}
\mu _v(z_v)=\mu _T(\psi _v(z^v))
=\mu ^{\scriptop{toric}}_{\lambda}(\Psi _v(\zeta ^v)),
\label{mutoricv}
\end{equation}
and $z^v={\theta _v}^{-1}(\zeta _v)$ is given in terms of $\zeta ^v$ by 
\begin{equation}
z^v_f=\zeta ^v_f
\,\prod_{f'\in F\setminus F_v}\, r_{f'}(\xi )^{-(v)^f_{f'}},
\quad f\in F_v, 
\label{zetaz}
\end{equation}
where $\xi$ is the element of $\Delta$ equal to the right hand side 
of (\ref{mutoricv}).  

\label{zzetalem}
\end{lemma}
\begin{proof}
It follows from Lemma \ref{Zvlem} and the paragraph 
preceding Proposition \ref{fibvprop} that if $z^v\in\rho ^v(Z)$, 
then $z^v\in U_v$ if and only if $z^v_{f'}\neq 0$ for every 
$f'\in F\setminus F_v$. That is, the set $Z_v$ in 
Proposition \ref{fibvprop} is equal to $Z\cap\C ^F_v$. 
It therefore follows from Theorem \ref{idthm} that 
each $N_{\C}$\--orbit in the $N_{\C}$\--invariant subset 
$\C ^F_v$ of $\C ^F_{\Delta}$ intersects 
the $N$\--invariant subset $Z_v$ of $Z$ in an 
$N$\--orbit in $Z_v$, that is,  
\[
M^{\scriptop{toric}}_v
=P_v(\C ^F_v)
=\varpi (p_v(Z_v))
=\varpi (M_v).
\]

If $z^v\in U_v$, then Proposition \ref{fibvprop} implies that 
$s_v(z^v)_f=z^v_f$ for every $f\in F_v$ and 
\[
s_v(z^v)_{f'}=r_{f'}(\mu _v(z^v)),\quad f'\in F\setminus F_v.
\]
If we define $t\in\T ^F_{\C}$ by 
\begin{eqnarray*}
t_{f'}&=&r_{f'}(\mu _v(z^v))^{-1},\quad f'\in F\setminus F_v, \\
t_f&=&\prod_{f'\in F\setminus F_v}\, 
r_{f'}(\mu _v(z^v))^{(v)^f_{f'}},\quad f\in F_v,
\end{eqnarray*}
then $(t\cdot s_v(z^v))_{t'}=1$ for every $t'\in F\setminus F_v$ and, 
for every $f\in F_v$,  $\zeta ^v_f:=(t\cdot s_v)_f$ is equal to 
the right hand side of (\ref{zzeta}). That is, 
$t\cdot s_v(z^v)=S_v(\zeta ^v)$, see the definition of $S_v$ 
in the paragraph preceding (\ref{Uvwtoric}).  On the other hand, 
it follows from (\ref{NC}) that $t\in N_{\C}$, and therefore 
\[
\Psi _v(\zeta _v)=P_v(t\cdot s_v(z^v))=P_v(s_v(z^v))
=\varpi\circ p_v(s_v(z^v))
=\varpi\circ\psi _v(z^v),
\]
that is, $\zeta ^v={\Psi _v}^{-1}\circ\varpi\circ\psi _v(z^v)$. 
\end{proof}

\begin{corollary}
Let $s$ be the relative interior of a face of $\Delta$. 
Then ${\mu _T}^{-1}(s)$ is equal to a stratum 
$S$ of the orbit type stratification in $M$ of the $T$\--action, 
and also equal to the preimage under $\varpi :M\to M^{\scriptop{toric}}$ 
of a $T_{\C}$\--orbit in $M^{\scriptop{toric}}$. 
If $s=\{v \}$ for a vertex $v$, then ${\mu _T}^{-1}(s)=\{ m_v\}$ 
for the unique fixed point $m_v$ in $M$ for the $T$\--action 
such that $\mu _T(m_v)=v$. 

The mapping $s\mapsto{\mu _T}^{-1}(s)$ 
is a bijection from the set $\Sigma _{\Delta}$ of all relative interiors of 
faces of $\Delta$ onto the set $\Sigma$ of all strata of the 
orbit type stratificiation in $M$ for the action of $T$. 
If $s,\, s'\in\Sigma _{\Delta}$ 
then $s$ is contained in the closure of $s'$ in $\Delta$ if and only 
if ${\mu _T}^{-1}(s)$ is contained in the closure of 
${\mu _T}^{-1}(s')$ in $M$. 

The domain of definition $M_v$ of $\varphi _v$ in $M$ is equal to 
the union of the $S\in\Sigma$ such that $m_v$ belongs to the closure of 
$S$ in $M$. The domain of definition $M_v^{\scriptop{toric}}=
\varpi (M_v)$ of $\Phi _v$ is equal to the union of the 
corresponding strata of the $T$\--action in $M^{\scriptop{toric}}$, 
each of which is a $T_{\C}$\--orbit in $M^{\scriptop{toric}}$. 
$M_v$ and $M_v^{\scriptop{toric}}$ are open cells in 
$M$ and $M^{\scriptop{toric}}$, respectively.   
\label{stratacor}
\end{corollary}
\begin{proof}
There exists a vertex $v$ of $\Delta$ 
such that $v$ belongs to the closure of $s$ in 
${\got t}^*$, which implies that $s$ is disjoint from 
all $f'\in F\setminus F_v$. Let $F_{v,\, s}$ denote the 
set of all $f\in F_v$ such that $s\subset f$, where 
$F_{v,\, s}=\emptyset$ if and only if $s$ is the 
interior of $\Delta$. 
For any subset $G$ of $F_v$, let 
$\C ^{F_v}_G$ denote the set of all 
$z\in\C ^{F_v}$ such that $z_f=0$ if $f\in G$ and 
$z_f\neq 0$ if $f\in F_v\setminus G$. It follows from 
$\mu _v=\mu _T\circ\psi _v$ and  (\ref{mu-1f}) 
that $\psi _v^{-1}({\mu _T}^{-1}(s))$ is equal to 
$U_v\cap\C ^{F_v}_G$ with $G=F_{v,\, s}$. 
The diffeomorphism $\theta _v$ maps this set onto 
the set $\C ^{F_v}_G$ with $G=F_{v,\, s}$. 
Because the sets of the form $\C ^{F_v}_G$ 
with $G\subset F_v$ are the strata of the 
orbit type stratification of the $\T ^{F_v}$\--action 
on $\C ^{F_v}$, and also equal to the $(\T _{\C})^{F_v}$\--orbits 
in $\C ^{F_v}$, the first statement of the corollary follows.  

The second statement follows from ${\mu _v}^{-1}(\{ v\})=\{ 0\}$ 
and the fact that $0$ is the unique fixed point of the 
$\T ^{F_v}$\--action in $U_v$. 

If $s\in\Sigma _{\Delta}$ and $v\in V$, then $m_v$ belongs to the closure 
of ${\mu _T}^{-1}(s)$ if and only if $s$ is not contained 
in any $f'\in F\setminus F_v$. This proves the 
characterization of the domain of 
definition $M_v:=Z_v/N={\mu _T}^{-1}(\Delta _v)$ of $\varphi _v$.  
The last statement follows 
from the fact that $\Phi _v$ is a diffeomorphism from 
$M_v^{\scriptop{toric}}$ onto the vector space $\C ^{F_v}$, 
and $\varpi$ is a diffeomorphism from $M_v$ onto $M_v^{\scriptop{toric}}$. 
\end{proof}

\begin{remark}
If $v,\, w\in V$, then 
\[
\varphi _w\circ {\varphi _v}^{-1}
={\psi _w}^{-1}\circ\psi _v
= {\theta _w}^{-1}\circ ({\Psi _w}^{-1}\circ\Psi _v)\circ\theta _v
={\theta _w}^{-1}\circ (\Phi _w\circ {\Phi _v}^{-1})\circ\theta _v.  
\]
Using the formula (\ref{coordtranstoric}) for 
$\Phi _w\circ {\Phi _v}^{-1}$, this can be used 
in order to obtain the formulas  (\ref{inFv}), (\ref{notinFv}) as a 
consequence of (\ref{zzeta}). In the proof, 
it is used that $\xi :=\mu _v(z^v)=\mu _w(z^w)$, 
$|z^v_f|=r_f(\xi )$ if $f\in F_v\setminus F_w$, and 
\[
\sum_{f\in F_v}\, (v)^f_{f'}\, (w)^g_f=(w)^g_{f'}
\]
if $f'\in F\setminus F_v$ and $g\in F_w$. 
\label{zzetarem}
\end{remark}

In the following corollary we describe the symplectic form 
$\sigma ^{\scriptop{toric}}_{\lambda}$ on the toric variety 
$M^{\scriptop{toric}}$  
in the toric variety coordinates. 
\begin{corollary}
For each $v\in V$, the symplectic form 
$({\Phi _v}^{-1})^*(\sigma ^{\scriptop{toric}}_{\lambda})$ 
on $\C ^{F_v}$ 
is equal to $({\theta _v}^{-1})^*(\sigma _v)$, where $\sigma _v$ 
is the standard symplectic form on $\C ^{F_v}$ given by (\ref{sigmav}). 
\label{zzetacor}
\end{corollary}
Because $r_{f'}(\mu _v(z^v))^2$ is an inhomogeneous linear function of 
the quantities $|z_f^v|^2$, it follows from (\ref{zzeta}) that 
the equations which determine the $|z^v_f|^2$ in terms 
of the quantities $|\zeta _f^v|^2$ are $n$ polyomial 
equations for the $n$ unknowns $|z^f_v|^2$, $f\in F_v$, 
where the coefficients of the polyomials are inhomogeneous 
linear functions of the $|\zeta _f^v|$, $f\in F_v$. In this sense the 
$|z_f^v|^2$, $f\in F_v$, are algebraic functions 
of the $|\zeta ^v_f|^2$, $f\in F_v$, and substituting these in 
(\ref{zzeta}) we obtain that the diffeomorphism 
${\theta _v}^{-1}$ from $\C ^{F_v}$ onto $U_v$ is an 
algebraic mapping.  If $\Delta$ is 
a simplex, when $M^{\scriptop{toric}}$ is the $n$\--dimensional 
complex projective space, we have an explicit formula 
for ${\theta _v}^{-1}$, see Subsection \ref{projsubsec}. 
However, already in the case that 
$\Delta$ is a planar quadrangle, when 
$M^{\scriptop{toric}}$ is a complex two\--dimensional 
Hirzebruch surface, we do not have an explicit formula 
for ${\theta _v}^{-1}$. See Subsection \ref{hirzebruchsubsec}. 

Summarizing, we can say that in the toric variety coordinates 
the complex structure is the standard one and the coordinate 
transformations are the relatively simple Laurent monomial 
transformations (\ref{coordtranstoric}). However, in the 
toric variety coordinates the $\lambda$\--dependent 
symplectic form in general is given by quite complicated 
algebraic functions. On the other hand, in the 
reduced phase space coordinates the symplectic form 
is the standard one, but the coordinate transformations 
(\ref{inFv}), (\ref{notinFv}) are more complicated. Also 
the complex structure in the reduced phase space 
coordinates, which depends on $\lambda$, 
is given by more complicated formulas. 
\begin{remark}
It is a challenge to compare the formula in Corollary \ref{zzetacor} 
for the symplectic form in toric variety coordinates with 
Guillemin's formula in \cite[Th. 3.5 on p. 141]{g}
and \cite[(1.3)]{gpaper}. Note that in the latter the pullback 
by means of the momentum mapping appears of a function 
on the interior of $\Delta$, where in general we do not have 
a really explicit formula for the momentum mapping in 
toric variety coordinates. 
\label{gformularem}
\end{remark}
\section{Examples}
\label{examplesec}
\subsection{The complex projective space}
\label{projsubsec}
Let $\Delta$ be an $n$\--dimensional simplex in ${\got t}^*$. 
A little bit of puzzling shows that there is a $\Z$\--basis 
$e_i$, $1\leq i\leq n$, of the integral lattice ${\got t}_{\Z}$ in ${\got t}$, 
such that, with the notation  
\begin{equation}
e_0=\, -\sum_{i=1}^n\, e_i,
\label{e0}
\end{equation}
the $X_f$, $f\in F$, are the $e_i$, $0\leq i\leq n$. 
That is, in the sequel we write $F=\{ 0,\, 1,\,\ldots ,\, n\}$.  
The Delzant simplex (\ref{Deltaineq}) is determined by the 
inequalities $\langle e_i,\,\xi\rangle +\lambda _i\geq 0$, 
$0\leq i\leq n$, which has a non\--empty interior if and only if 
\begin{equation}
\gamma :=\sum_{i=0}^n\,\lambda _i>0. 
\label{sumpos}
\end{equation}
In the sequel we take for $v$ the vertex determined by the 
equations $\langle e_i,\,\xi\rangle +\lambda _i=0$ for all $1\leq i\leq n$, 
where $F_v=\{ 1,\,\ldots ,\, n\}$. 
If we write $\xi _i=\langle e_i,\,\xi\rangle$, 
$1\leq i\leq n$, when $\xi\in {\got t}^*$, then (\ref{muv}) yields that 
\[
\mu _v( z^v)_i=|z_i|^2/2-\lambda _i,\quad 1\leq i\leq n.
\]
It follows from (\ref{rv}) that 
\[
r_0(\xi )=(2(-\sum_{i=1}^n\,\xi _i+\lambda _0))^{1/2},
\]
and therefore (\ref{zzeta}) yields that 
\begin{equation}
\zeta ^v_i=z^v_i\, (2\gamma -\| z^v\| ^2)^{-1/2},\quad 1\leq i\leq n,
\label{zzetaproj}
\end{equation}
where we have written 
\[
\| z^v\| ^2=\sum_{i=1}^n\, |z^v_i|^2. 
\]
Note that $U_v$ is the open ball in $\C ^n$ with center at the origin 
and radius equal to $(2c)^{1/2}$. 

The equations (\ref{zzetaproj}) 
imply that 
\[
\| \zeta ^v\| ^2=\| z_v\| ^2/(2\gamma -\| z^v\| ^2),
\]
hence 
\[
\| z^v\| ^2=2\gamma \,\|\zeta ^v\| ^2/(1+\|\zeta ^v\| ^2). 
\]
Therefore the mapping ${\theta ^v}^{-1}:\zeta ^v\mapsto z^v$ 
is given by the explicit formulas 
\begin{equation}
z^v_i=\zeta ^v_i\, (2\gamma /(1+\|\zeta ^v\| ^2))^{1/2},\quad 1\leq i\leq n. 
\label{zetazproj}
\end{equation}
It can be verified that the symplectic form 
$({\theta _v}^{-1})^*(\sigma _v)$, where 
\[
\sigma _v=(1/2\pi )\,\sum_{i=1}^n\,\op{d}\! x^v_i\wedge\op{d}\! y^v_i
\]
is the standard symplectic form in (\ref{sigmav}), is equal to 
$\gamma$ times the Fubini\--Study form in 
Griffiths and Harris \cite[p. 30, 31]{gh}. 
In view of Remark \ref{cohomrem} this 
agrees with the fact that the de Rham 
cohomology class of the Fubini\--Study form is Poincar\'e dual to 
the homology class of a complex projective hyperplane 
in the complex projective space, see 
Griffiths and Harris \cite[p. 122]{gh}. 
\subsection{The Hirzebruch surface}
\label{hirzebruchsubsec} 
Let $n=2$ and let $\Delta$ be a quadrangle in the ${\got t}^*$ plane. 
A little bit of puzzling shows that there is an 
$m\in\Z _{\geq 0}$ and a $\Z$\--basis 
$e_1$, $e_2$ of the integral lattice ${\got t}_{\Z}$ in ${\got t}$, 
such that the $X_f$, $f\in F$, are the $e_i$, $1\leq i\leq 4$, 
with $e_3= \, -e_1+m\, e_2$, and $e_4=\, -e_2$. 
We recognize the toric variety 
$M^{\scriptop{toric}}$ as the Hirzebruch surface 
$\Sigma _m$, see Hirzebruch 
\cite{h}. 

The Delzant polytope (\ref{Deltaineq}) is determined by the 
inequalities $\langle e_i,\,\xi\rangle +\lambda _i\geq 0$, 
$1\leq i\leq 4$, which is a quadrangle if and only if 
\begin{equation}
\gamma _{\pm}:=\lambda _1+\lambda _3\pm m\,\lambda _4>0,\quad
\label{sumpmpos}
\end{equation}
 which inequalities imply that $\lambda _2+\lambda _4
=\gamma _++\gamma _->0$.   

In the sequel we take for $v$ the vertex determined by the 
equations $\langle e_i,\,\xi\rangle +\lambda _i=0$ for $i=1,\, 2$, 
where $F_v=\{ 1,\, 2\}$. 
If we write $\xi _i=\langle e_i,\,\xi\rangle$, 
$1\leq i\leq 2$, when $\xi\in {\got t}^*$, then (\ref{muv}) yields that 
\[
\mu _v( z^v)_i=|z_i|^2/2-\lambda _i,\quad 1\leq i\leq 2.
\]
It follows from (\ref{rv}) that 
\[
r_3(\xi )=(2(-\xi _1+m\,\xi _2+\lambda _3))^{1/2},
\quad
r_4(\xi )=(2(-\xi _2+\lambda _4))^{1/2}
\]
and therefore (\ref{zzeta}) yields that 
\begin{eqnarray}
\zeta ^v_1&=&z^v_1\, (2\gamma _-|z^v_1|^2 +m\, |z^v_2|^2)^{-1/2},
\label{zzeta1}\\
\zeta ^v_2&=&z^v_2\, (2\gamma _-|z^v_1|^2 +m\, |z^v_2|^2)^{m/2}
\, (2(\gamma _++\gamma _-)-|z^v_2|^2)^{-1/2}.
\label{zzeta2}
\end{eqnarray}
If we write $t_i=|z^v_i|^2$ and $\tau _i=|\zeta ^v_i|^2$, 
then this leads to the equations 
\begin{eqnarray*}
\tau _1&=&t_1/(2\gamma_- -t_1+m\, t_2),\\
\tau _2&=&t_2\, (2\gamma _- +t_1+m\, t_2)^m
/(2(\gamma _+ +\gamma _-)-t_2)
\end{eqnarray*}
for $t_1,\, t_2$. If we solve $t_1$ from the first equation, 
\[
t_1=(2\gamma _- +m\, t_2)\,\tau _1/(1+\tau _1),
\]
and substitute this into the second equation, then 
this leads to the polynomial equation 
\begin{equation}
(1+\tau _1)^m\,\tau _2\, 
(2(\gamma _+ +\gamma _-) -t_2)=t_2\, (2\gamma _- +m\, t_2)^m
\label{t2eq}
\end{equation}
of degree $m+1$ for $t_2$. If we substract the left hand side from the 
right hand side then the derivative with respect 
to $t_2$ is strictly positive, and one readily obtains 
that for every $\tau _1,\,\tau _2\in\R _{\geq 0}$ there 
is a unique solution $t_2\in\R _{\geq 0}$, confirming the first statement in 
Lemma \ref{zzetalem}. 

On the other hand, if we work over $\C$, and view both the 
parameter  $\varepsilon := (1+\tau _1)^m\,\tau _2$ and 
the unknown $t_2$ as elements of the complex projective 
line $\Proj ^1$, then the equation (\ref{t2eq}) defines 
a complex algebraic curve $C$ in the $(t_2,\, \varepsilon )$\--plane 
$\Proj ^1\times\Proj ^1$, where the restriction to $C$ 
of the projection to the first variable $t_2$ is a complex analytic 
diffeomorphism from $C$ onto $\Proj ^1$, as on $C$ we have 
that $\varepsilon$ is a complex analytic function of $t_2$. 
In particular $C$ is irreducible. The restriction to $C$ 
of the projection to the second variable $\varepsilon$ is 
an $(m+1)$\--fold branched covering. Over $\varepsilon =0$ 
and over $\varepsilon =\infty$ we have that $m$ of the $m+1$ 
branches come together, whereas there are two more branch points 
on the $\varepsilon$\--line over which only two of the branches 
come together. The fact that $C$ is irreducible implies that 
the part of $C$ over the complement of the branch points 
is connected, and therefore the analytic continuation of any  
solution $t_2$ of (\ref{t2eq}), as a complex analytic function 
of $\varepsilon$ in the complement of the branch points, 
will reach each other branch if $\varepsilon$ runs over 
a suitable loop. In other words, the 
solution $t_2$ is an algebraic function of $\varepsilon$ 
of degree $m+1$, and no branch of a solution is 
of lower degree. This holds in particular 
for our solutions $t_2\in\R _{\geq 0}$ for $\varepsilon\in\R _{\geq 0}$.

\noindent
J.J. Duistermaat\\
Mathematisch Instituut, Universiteit Utrecht\\
P.O. Box 80 010, 3508 TA Utrecht, The Netherlands\\
e\--mail: duis@math.uu.nl

\bigskip\noindent
A. Pelayo\\
Department of Mathematics, University of Michigan\\
2074 East Hall, 530 Church Street, Ann Arbor, MI 48109--1043, USA\\
e\--mail: apelayo@umich.edu

\end{document}